\documentclass{elsarticle}
\usepackage{amsmath,amsthm,amssymb}

\newtheorem{theo}{Theorem}
\newtheorem{lemm}[theo]{Lemma}
\newtheorem{prop}[theo]{Proposition}
\newtheorem{coro}[theo]{Corollary}
\theoremstyle{remark}
\newtheorem*{rem}{\bf Remark}
\newtheorem*{exa}{\bf Example}
    \makeatletter
    \def\ps@pprintTitle{%
      \let\@oddhead\@empty
      \let\@evenhead\@empty
      \let\@oddfoot\@empty
      \let\@evenfoot\@oddfoot
    }
    \makeatother

\begin{document}

\title{Bounds for the Euclidean minima of function fields}

\author[1]{Piotr Maciak}
\ead{piotr.maciak@epfl.ch}

\author[2]{Marina Monsurr\`o}
\ead{marina.monsurro@gmail.com}

\author[3]{Leonardo Zapponi}
\ead{zapponi@math.jussieu.fr}

\address[1]{\'Ecole Polytechnique F\'ed\'erale de Lausanne (EPFL), Lausanne (CH)}
\address[2]{Univerist\`a Europea di Roma (UER), Roma (I)}
\address[3]{Universit\'e Pierre et Maie Curie (UPMC), Paris (F)}

\begin{abstract} In this paper, we define Euclidean minima for function fields and give some bound for this invariant. We furthermore show that the results are analogous to those obtained in the number field case.
\end{abstract}

\begin{keyword} Function fields\sep algebraic curves\sep Euclidean minima

\end{keyword}

\maketitle
\tableofcontents
\newpage
\section{Introduction}

The Euclidean minimum is a numerical invariant which measures how elements of a number field can be approximated by algebraic integers. Its study is a classical topic in algebraic number theory, going back to Minkowski and Hurwitz. Many mathematicians have studied its properties; among the others, we can cite Barnes, Swinnerton-Dyer, Cassels, Davenport and Fuchs, but this list is far from being exhaustive. This research area is still very active nowadays, and many problems are still open. We can mention for example the works of Van der Linden, Cerri and Bayer-Fluckiger. One of the central questions is to find upper bounds for Euclidean minima. We refer to~\cite{l} for a nice survey of the topic.

The Euclidean minima can also be defined in the function field case, replacing the absolute value by the degree. Not much is known in this situation. The existence of Euclidean function fields was studied by Armitage, Markanda, Madam, Queen and Smith but Euclidean minima never explicitly appear in their work. The aim of the present paper is to investigate this question more closely, trying to obtain bounds analoguous to  the known bounds in the number field case.

The paper is organized as follows: section 2, which is of geometric flavour is mainly inspired from the guiding example of polynomials: it is well-known that, given a field $k$, the ring $k[T]$ is euclidean with respect to the degree. Now, from a geometric point of view, the field $k(T)$ is just the function field of the projective line, and $k[T]$ is the subring consisting of functions which are regular outside the point at infinity. Starting from this observation, we try to generalize the construction: we first of all consider a (smooth, proper)  curve $\mathcal C$ over a field $k$, which is assumed to be algebraically closed. We then consider a finite subset $S$ of $\mathcal C$ and define a degree function, denoted by $\deg_S$, for rational functions on the curve. The ring $\mathcal O_S$ is the subset of the field $K=k(\mathcal C)$ consisting of rational functions which are regular outside $S$. In other words, the affine curve $\mathcal C\setminus S$ is just the spectrum of $\mathcal O_S$. For $\mathcal C=\Bbb P^1$ and $S=\{\infty\}$, we then find $K=k(T)$ and the ring $\mathcal O_S=k[T]$ is an univariate polynomial ring over $k$, the function $\deg_S$ coinciding with the usual degree for polynomials. The (logarithmic) Euclidean minimum of $K$ with respect to $S$, denoted by $M_S(K)$, is then defined as
\[M_S(K)=\max_{x\in K}\min_{y\in\mathcal O_S}\{\deg_S(x-y)\}.\]
This integer measures how well a rational function can be approximated by an element of $\mathcal O_S$. In particular, the ring $\mathcal O_S$ is Euclidean with respect to the function $\deg_S$ if and only if the inequality $M_S(K)<0$ holds. The main result of the section relates $M_S(K)$ to a geometric invariant we introduce here: the degree of speciality $\mu(S)$ of $S$, which, roughly speaking, describes the behaviour of differential forms on $\mathcal S$ having no poles outside $S$. In Theorem 2, we prove that we have the inequality $M_S(K)\leq\mu(S)$. The main ingredient of the proof is the Riemann-Roch theorem. It then follows that $M_K(S)$ is less than or equal to $2g-1$, where $g$ denotes the genus of $\mathcal C$ (this bound again follows from Riemann-Roch formula), providing an upper bound which can be considered, as we will show later, as the precise analogue of the results obtained in the number field case in~\cite{b}. The end of the section is devoted to the special case where the set $S$ is a singleton. In this situation, we can prove (cf. Proposition 3) that we have in fact the equality $M_S(K)=\mu(S)$ and, moreover, that $\mu(S)$ is the greatest integer appearing in the Weierstrass gap sequence of $S$. In particular, we obtain the lower bound $M_S(K)\geq g$ and we provide examples where $M_S(K)$ is explicitly computed (hyperelliptic curves, \'etale covers of the affine line, classical curves).

As stated above, in section 3 we show that Theorem 2 can be considered as a (logarithmic) analogue of the general bounds obtained in the number field case. It is important to notice that there is a crucial difference between number fields an function fields. Indeed, a number field can be obtained in a unique way as an extension of the rationals. In particular, its degree and discriminant are uniquely defined global objects. For a function field $K$, things are quite different, since it can be realized in many (infinite) ways as an extension of the field $k(T)$. Its degree and discriminant are no more intrinsic invariants and strongly depend on this realization. This is the main reason why in section 2 we tried to obtain results avoiding this arbitrary choice. But now, if we want to find a parallel between the present results and the number field case couterpart, we have to consider the field $K$ as a finite extension of the field $k(T)$. In this setting, we slightly modify the definition of the Euclidian minimum, now simply denoted by $M(K)$, which turns out to be just a special case of the one defined before (for a particular choice of the set $S$). In theorem 8, we obtain an upper bound for $M(K)$ only depending on the degree and on the discriminant of the extension $\mathcal O_K/k[T]$ (here, the ring $\mathcal O_K$ is the integral closure of $k[T]$ in $K$, which, once again coincides withnthe ring $\mathcal O_S$ previously defined). In fact, Theorem 8 is just a weaker version (and a direct consequence) of Theorem 2, the main ingredient for its proof being the Hurwitz formula. The main limitation is that we have to suppose that the ramification above the point at infinity is tame. A more general, and sharper result could be obtained by considering the full ramification of the extension $K/k(T)$ (taking account of the behaviour above infinity) but we decided to state it in a  form more similar to the number field case.

Until now, we worked in a geometric setting, assuming that the base field $k$ is algebraically closed. In the last section of the paper, we drop this assumption and investigate the more general case of a perfect base field. Most of the constructions and results remain valid in this situation: the Euclidean minimum and the index of speciality are defined in the same way and the inequality of Theorem 2 still holds. However, it is no longer true that $\mu(S)$ is less than or equal to $2g-1$. Moreover, it actually turns out that the Euclidean minimum depends on the base field. Indeed, we explicitly treat the case of hyperelliptic curves and in Theorem 10 we show that in this case $M(K)$ can actually be equal to $2g$, which is impossible if $k$ is algebraically closed.

\

We would like to thank Eva Bayer-Fluckiger, who originally motivated and then constantly encouraged this joint work. Her results in the number field case have been the main inspiration of the paper. The second and third authors also express their thanks for her invitation at the EPFL, where most of the results were obtained.

\section{Geometric approach}
\subsection{Settings and notation} In the following, $\mathcal C$ denotes a smooth, projective curve of genus $g$, defined over an algebraically closed field $k$. Set $K=k(\mathcal C)$ and let $v_P$ (resp. $\mathcal O_P$) be the valuation of $K$ (resp. the local ring) associated to a point $P$ of $\mathcal C$. For any finite, nonempty subset $S$ of $\mathcal C$, let $\mathcal O_S$ be the subring of $K$ consisting of rational functions regular outside $S$. For any rational function $x\in K^\times$,  consider the integer $\deg_S(x)$ defined by
\[\deg_S(x)=-\sum_{P\in S}v_P(x).\]
For simplicity, we furthermore set $\deg_S(0)=-\infty$. It is easily checked that for any element $x\neq0$ of $\mathcal O_S$, we have the inequality $\deg_S(x)\geq0$, which is an equality if and only if $x$ belongs to $\mathcal O_S^\times$. By construction, for any $x,y\in K$, we have the relations
\begin{itemize}
\item $\deg_S(xy)=\deg_S(x)+\deg_S(y)$,
\item $\deg_S(x+y)\leq\max\{\deg_S(x),\deg_S(y)\}$.
\end{itemize}

\begin{rem} For $\mathcal C=\Bbb P^1$ and $S=\{\infty\}$, the ring $\mathcal O_S$ can be identified with the polynomial ring $k[T]$ and, for any $x\in\mathcal O_S$, the integer $\deg_S(x)$ is the usual degree of a polynomial.
\end{rem}

The {\it Euclidean minimum} of $K$ (with respect to $S$) is the integer defined as
\[M_S(K)=\max_{x\in K^\times}\min_{y\in\mathcal O_S}\{\deg_S(x-y)\}.\]

The inequality $M_S(K)<0$ holds if and only if the ring $\mathcal O_S$ is Euclidean with respect to the Euclidean function $\deg_S$.

\subsection{The index of speciality} Following the usual notation, for any divisor $D$ on $\mathcal C$, let $\Omega(D)$ be the $k$-vector space of differential forms $\omega$ on $\mathcal C$ satisfying the relations $v_P(\omega)+v_P(D)\geq0$ for any $P\in\mathcal C$ (in~\cite{s}, this $k$-vector space is denoted by $\Omega(-D)$, we nevertheless adopt the more recent and usual notation). Denote by $\Bbb Z[S]$ the free abelian group generated by the elements of $S$, i.e. the subgroup of $\mbox{Div}(\mathcal C)$ consisting of divisors whose support is contained in $S$. The {\it index of speciality} of $S$ is the integer $\mu(S)$ defined by
\[\mu(S)=\min\{\deg(D)\,\,|\,\,D\in\Bbb Z[S]\mbox{ and }\Omega(-D)=0\}.\]

\begin{rem} The dimension of the $k$-vector space $\Omega(D)$ only depending on the linear equivalence of the divisor $D$, we may replace $\Bbb Z[S]$ by the subgoup of $\mbox{Pic}(\mathcal C)$ generated by $S$.
\end{rem}

\begin{lemm} We have the inequalities
\[g-1\leq\mu(S)\leq2g-1.\]
\end{lemm}

\begin{proof} Let $D$ be a divisor on $\mathcal C$ of degree greater than or equal to  $2g-1$ and suppose that there exist an element $\omega\in\Omega(-D)$ different from $0$. In this case, we obtain the relations
\[2g-2=\deg(\omega)\geq\deg(D)\geq2g-1,\]
a contradiction. We therefore have $\Omega(-D)=0$. Since the field $k$ is algebraically closed, there exists an element $D\in\Bbb Z[S]$ of degree $2g-1$, which implies that $\mu(S)\leq2g-1$. Suppose now that $\deg(D)<g-1$ and denote by $L(D)$ the $k$-vector space of rational functions $x\in K$ for which $v_P(x)+v_P(D)\geq0$. The Riemann-Roch formula leads to the relations
\[\dim_k\Omega(-D)=\dim_kL(D)-\deg(D)+g-1>0,\]
and the second inequality follows.
\end{proof}

\subsection{An upper bound for Euclidean minima}

We now prove the main result of the section. We will later see that (and explain why) it is analogous to the bound obtained in the number field case.

\begin{theo} We have the inequality
\[M_S(K)\leq\mu(S).\]
\end{theo}

\begin{proof} Let $A_K$ denote the adele ring of $K$, i.e. an element $r\in A_K$ is a collection $r=(r_P)_{P\in\mathcal C}$ of rational functions such that $r_P\in\mathcal O_P$ for almost all $P$ in $\mathcal C$ (see~\cite{s}, Chapter II, \S5 for general definitions and properties of adele rings on curves). For any divisor $D$ on $\mathcal C$, let $A_K(D)$ denote the $k$-vector space of adeles $r$ such that $v_P(r_P)+v_P(D)\geq0$ for all $P\in\mathcal C$. By Serre duality, the quotient $I(D)=A_K/(A_K(D)+K)$ is canonically isomorphic to the dual of the $k$-vector space $\Omega(-D)$. Let now $D\in\Bbb Z[S]$ be a divisor such that $\deg(D)=\mu(S)$. Fix a rational function $x\in K^\times$ and consider the adele $r=(r_P)$ defined by
\[r_P=\begin{cases}
x\quad\mbox{ for }P\in S,\\
0\quad\mbox{  otherwise.}
\end{cases}\]
By construction, we have $\Omega(D)=0$ and therefore $I(D)=0$. In particular, there exist $y\in K$ and $r'\in A_K(D)$ such that $r=r'+y$. It then follows that the poles of the rational function $y$ are contained in $S$, and thus we have $y\in\mathcal O_S$. Furthermore, from the definition of $r$, we have the inequalities
\[v_P(x-y)+v_P(D)\geq0\]
for any $P\in S$, which finally leads to
\[\deg_S(x-y)\leq\deg(D)=\mu(S).\]
We therefore have the inequality $M_S(K)\leq\mu(S)$.
\end{proof}

\subsection{A special case. Weierstrass gaps}

Though we are not able to prove it in full generality, it is reasonable to believe that the inequality of Theorem 2 is in fact an equality. We will now see that it is actually true if the set $S=\{P\}$ is reduced to a single point. Until the end of this section, we restrict to this situation. The study of the index of speciality, simply denoted by $\mu(P)$ in this case, is closely related to the {\it Weierstrass gap sequence} of $P$, we therefore briefly recall some basic facts of the theory: setting $\ell(nP)=\dim_kL(nP)$ and $i(nP)=\dim_k\Omega(nP)$, we have the relations
\[\ell(0)=1,\ell(nP)\leq\ell((n+1)P)\leq\ell(nP)+1\mbox{ and }\ell((2g-1)P)=g.\]
Furthermore, for $n\geq{2g-1}$, we have the identity $\ell(nP)=n-g+1$. In terms of differential forms, the Riemann-Roch formula leads to the relations
\[i(0)=g,i(-nP)\geq i(-(n+1)P)\geq i(-nP)-1\mbox{ and }i((1-2g)P)=0.\]
A {\it Weierstrass gap} of $P$ is an integer $n$ such that $l(nP)=l((n-1)P)$, which means that there does not exist a rational function $x\in K$ having a unique pole at $P$ of  exact order $n$. Equivalently, we have the identity $i(-nP)=i((1-n)P)-1$, which implies that there exists a regular differential form having a zero at $P$ of order $n-1$. It then follows that there are exactly $g$ gaps $0<n_1<n_2\cdots<n_g<2g$. The gap sequence is just the collection $(n_1,\dots,n_g)$. Notice that, by definition, the index of speciality $\mu(P)$ is just the integer $n_g$, the greatest Weierstrass gap. Given the curve $\mathcal C$, there exist a sequence $N=(n_1,\dots,n_g)$ such that all but finitely many points of $\mathcal C$ have $N$ as Weierstrass gap sequence. The points having a Weierstrass gap sequence different from $N$ are called {\it Weierstrass points}, while the others are {\it ordinary points}. If $N=\{1,2,\dots,g\}$, we say that $\mathcal C$ is a {\it classical curve}. In characteristic $0$, any curve is classical, but this is false in positive characteristic (though it is the case if the characteristic of the base field is large enough with respect to the genus of the curve).

\begin{prop} If $S=\{P\}$ is reduced to a single point, then we have the identity
\[M_S(K)=\mu(S).\]
In particular, we have the inequalities
\[g\leq M_S(K)\leq2g-1.\]
\end{prop}

\begin{proof} Given a point $Q\neq P$ of $\mathcal C$, the Riemann-Roch theorem implies that there exists a rational function $x\in K$ having no poles outside the set $\{P,Q\}$ and such that $v_P(x)=-\mu(P)$, so that $\deg_S(x)=\mu(S)$. Now, the relation $\deg_S(x-y)<\mu(x)$, with $y\in\mathcal O_S$, would imply that the rational function $y$ has a unique pole at  $P$ of order $\mu(P)$, which is impossible since $\mu(P)$ is a Weierstrass gap. The two inequalities again follow from the fact that $\mu(P)$ is a Weierstrass gap.
\end{proof}

\begin{exa} Once again, taking $\mathcal C=\Bbb P^1$ and $S=\{\infty\}$, so that $K=k(T)$ and $\mathcal O_S=k[T]$, we find the inequality $M_S(K)<-1$, recovering the well known fact that the ring $k[T]$ is Euclidean with respect to the degree.
\end{exa}

We close this section with three examples where we can actually compute the Euclidean minimum, and show that for fixed $g$, the bounds of Proposition 3 are attained.

\begin{coro} Let $\mathcal C$ be an hyperelliptic curve and suppose that $P$ is a Weierstrass point on it. We then have the identity
\[M_S(K)=2g-1.\]
\end{coro}

\begin{proof} A well-know result on hyperelliptic curves asserts that $P$ is a Weierstrass point if and only its gap sequence is $(1,3,\dots,2g-1)$ (see for example~\cite{f}, Exercice 8.37, where the characteristic of the base field is assumed to be $0$, but everything remains true in full generality), from which we deduce the relation $\mu(P)=2g-1$. The result is then a direct consequence of Proposition 3.
\end{proof}

\begin{rem} The above result is actually true for superelliptic curves, i.e. curves $\mathcal C$ defined by affine equations $Y^n=f$, where $f\in k[X]$ and the characterictic of $k$ does not divide $n$. More precisely, if $r$ denotes the degree of $f$, then there is a canonical degree $n$ cover $\mathcal C\to\Bbb P^1$ totally ramified at $r$ or $r+1$ points. If $P$ is one of them, it is then easily checked that $2g-1$ belongs to its  gap sequence and is therefore the greatest Weierstrass gap. In particular, setting $S=\{P\}$, we find $\mu(S)=2g-1$ and Proposition 3 asserts that we have the identity $M_S(K)=2g-1$.
\end{rem}

\begin{coro} Suppose that there exists an \'etale cover $\pi:\mathcal C\setminus S\to\Bbb A^1$. We then have the identity
\[M_S(K)=2g-1.\]
\end{coro}

\begin{proof} Consider a differential form $\omega$ on $\Bbb P^1$ having a double pole at $\infty$ (it is unique, up to multiplication by an element of $k^\times$). Since the canonical divisor of $\Bbb P^1$ has degree $-2$, it follows that $\omega$ is regular and non-vanishing outside $\infty$. In particular, the cover $\pi$ being \'etale outside $\infty$, the pull-back $\pi^*\omega$ is regular and nowhere vanishing outside $P=\pi^{-1}(\infty)$, so that $\omega$ has a unique zero at $P$, of order $2g-2$, and thus the integer $2g-1$ belongs to the gap sequence of $P$, which concludes the proof.
\end{proof}

\begin{coro} Suppose that the curve $\mathcal C$ is classical and that $P$ is not a Weierstrass point. We then have the identity
\[M_S(K)=g.\]
\end{coro}

\begin{proof} It is again a consequence of Proposition 3, since in this case, the gap sequence of $P$ is $(1,2,\dots,g)$, leading to $\mu(P)=g$.
\end{proof}

\section{The analogy with the number field case}

\subsection{Upper bounds for the Euclidean minima of number fields}  As it is stated in the previous paragraph, Theorem 2 could seem quite far from the kind of results obtained in the number field case. We will now show that it actually leads to an upper bound for the Euclidean minima which is, {\it mutatis mutandi}, a complete analogy with the known results and conjectures. We first of all briefly summarize the number field case: let $K$ be a number field, i.e. a finite extension of the field $\Bbb Q$ of rational numbers, and denote by $\mathcal O_K$ its ring of integer (the integral closure of $\Bbb Z$ in $K$). The {\it Euclidean minimum} of $K$ is the real number $M(K)$ defined as
\[M(K)=\sup_{x\in K}\inf_{y\in\mathcal O_K}\{|N_{K/\Bbb Q}(x-y)|\},\]
where $|\cdot|$ denotes the usual (Archimedean) absolute value. The study of Euclidean minima is a hard and classical topic in algebric number theory. One of the main results, obtained by E. Bayer-Fluckiger in~\cite{b} in 2006, is the following  general upper bound for Euclidean minima:

\begin{theo}[Bayer-Fluckiger, 2006] Let $n$ be the degree of the extension $K/\Bbb Q$ and denote by $d_k$ its discriminant. We then have the inequality
\[M(K)\leq2^{-n}|d_k|.\]
\end{theo}

\subsection{The function fields analogue}

The aim of this section is to give a bound for the Euclidean minima of function fields, analogous to Theorem 7. We first of all fix the notation and settings: the field $\mathcal Q=k(T)$ will play the role of the field $\Bbb Q$, the ring $\mathcal Z=k[T]$ replacing $\Bbb Z$. We denote by $K$ a finite, separable extension of $\mathcal Q$ of degree $n$ and by $\mathcal O_K$ the integral closure of $\mathcal Z$ in $K$. The discriminant of the extension $\mathcal O_K/\mathcal Z$, denoted by $d_K$, considered as an ideal of $\mathcal Z$, is generated by a polynomial $f$ and we set $\deg(d_K)=\deg(f)$. The {\it Euclidean minimum} $M(K)$ is the integer defined as
\[M(K)=\max_{x\in K}\min_{y\in\mathcal O_K}\{\deg(N_{K/\mathcal Q}(x-y))\},\]
where we have set $\deg=\deg_\infty$ (the extention to $\mathcal Q$ of the usual degree of a polynomial, cf. the first remark of \S2).
\begin{rem} Before continuing, we must stress on two important points:
\begin{itemize}
\item The rings $\Bbb Z$ and $\mathcal Z$ are both Euclidean. In the first case, the associated Euclidean function is the absolute value, which is multiplicative i.e. $|xy|=|x|\cdot|y|$, while in the latter case the Euclidean function is the degree, which is additive i.e. $\deg(xy)=\deg(x)+\deg(y)$. The Euclidean minimum is then defined in a similar way, but because of this difference, the bounds for function fields should be considered as a 'logarithmic' analogue of those obtained in the number field case.
\item For function fields case, the discriminant $d_K$ does not take account of the ramification above $\infty$ (the valuation associated to the prime element $T^{-1}$); this is the main difference with number fields, where the discriminant completely describes the ramification of the extension.
\end{itemize}
\end{rem}
We can now state and prove the main result of this section, which is a weaker form of Theorem 2 and can be considered as the analogue of Theorem 7.

\begin{theo} The assumptions and hypothesis being as above, suppose that the extension $K/\mathcal Q$ is tamely ramified above $\infty$. We then have the inequality
\[M(K)\leq\deg(d_K)-n.\]
\end{theo}

\begin{proof} From a geometric point of view, the extension $K/\mathcal Q$ corresponds to a finite cover $\pi:\mathcal C\to\Bbb P^1$. Setting $S=\pi^{-1}(\infty)$, the ring $\mathcal O_K$ then coincides with the ring $\mathcal O_S$ defined in the first section and, taking spectra, the extension $\mathcal O_K/\mathcal Z$ defines a finite cover $\mathcal C\setminus S\to\Bbb A^1$, which is just the restriction of $\pi$ to the open set $\Bbb A^1\subset\Bbb P^1$. For any rational function $x\in K$, a straightforward computation leads to the relation
\[\deg(N_{K/\mathcal Q}(x))=\deg_S(x).\]
In particular, we find the identity
\[M(K)=M_S(K).\]
Let $R$ be the ramification divisor of $\pi$ and $B=\pi_*R$ its branch divisor (see~\cite{h}, \S IV.2 for general definitions and properties). We then have the identity $\deg(R)=\deg(B)$ and $d_K$, considered as a divisor, is just the restriction of $B$ to $\Bbb A^1$. Moreover, if $g$ denotes the genus of $\mathcal C$, the Hurwitz formula (\cite{h}, Corollary 2.4 of Chapter IV) reads
\[2g-2=\deg(B)-2n.\]
Now, if the ramification above $\infty$ is tame, we then have the inequality
\[\deg(B)\leq\deg(d_K)+n-1.\]
In this case, combining Lemma 1 and Theorem 2, we finally obtain the relations
\[M(K)\leq2g-1=\deg(B)-2n+1\leq\deg(d_K)-n,\]
which concludes the proof.
\end{proof}

\begin{theo} Suppose that the extension $K/\mathcal Q$ is totally ramified above $\infty$ and that the characteristic of $k$ does not divide $n$. We then have the inequalities
\[\frac12\left(\deg(d_K)-n-1\right)\leq M(K)\leq\deg(d_K)-n.\]
\end{theo}

\begin{proof} Following the notation of the proof of Theorem 8, the ramification above $\infty$ being tame, we have the identity
\[\deg(B)=\deg(d_K)+n-1\]
Now, since the set $S$ is reduced to a point, Proposition 3 asserts that we have the inequalities
\[g-1\leq M(K)\leq2g-1\]
and the Hurwitz formula finally leads to the relations
\[\frac12\left(\deg(d_K)-n-1\right)\leq M(K)\leq\deg(d_K)-n,\]
as desired.
\end{proof}

\section{The case of a general base field}

\subsection{General base fields} Until now, we assumed that the base field $k$ is algebraically closed. We now drop this assumption and we only suppose that $k$ is perfect. Most of the constructions of the previous sections actually apply in this more general situation, with only minor changes: first of all, going back to section 2 and setting $G_k=\mbox{Gal}(\bar k/k)$, we must assume that the set $S\subset\mathcal C(\bar k)$ is stable under the action of $G_k$. The definition of the $S$-degree $\deg_S$ and of the Euclidean minima are exactly the same. Concerning the index of speciality, we just replace the group $\Bbb Z[S]$ by the subgroup $\Bbb Z[S]^{G_k}$ of $G_k$-invariants. In particular, it turns out that $\mu(S)$ can be strictly bigger than in the case of an algebraically closed field $k$.

\begin{exa} Suppose that $\mathcal C$ is an elliptic curve and that $S=\{P,Q\}$ consists of two $G_k$-conjugated points. Setting $D=P+Q$, we then find $\Bbb Z[S]^{G_k}=\Bbb Z[D]$ an therefore $\mu(S)=\deg(D)=2$. On the other hand, over $\bar k$, setting $F=P-Q\in\Bbb Z[S]$, we find $\Omega(-F)=0$ and it actually turns out that $\mu(S)=\deg(F)=0$.
\end{exa}

The result in Theorem 2 still holds. However, in this case we no longer have the inequality $\mu(S)\leq2g-1$, which was of crucial importance in section 3. Nevertheless, if $S$ is reduced to a point, then everything works perfectly.

\begin{rem} If the residual degrees (i.e. the degrees of the fields of definition) of the elements of $S$ are globally coprime, which is the same assumption made in~\cite{a}, then we still have the inequality $\mu(S)\leq2g-1$ and all the results of the previous sections are true.
\end{rem}

\subsection{An example: hyperelliptic curves}

We close the paper with an example where the Euclidean minimum can be explicitly computed, showing that its behaviour is different than in the case of an algebraically closed field. In the following, we assume that the characteristic of $k$ is different from $2$. Let $g$ be a positive integer, fix a polynomial $f\in k[X]$ of degree $2g+2$ whose leading coefficient is not a square in $k^\times$ and consider the hyperelliptic curve $\mathcal C$ defined by the affine equation $Y^2=f$. By construction, its genus is equal to $g$ and there exist two points $P$ and $Q$ at infinity which are permuted by $G_k$. There is a canonical double cover $\mathcal C\to\Bbb P^1$ and, setting $S=\{P,Q\}$, the ring $\mathcal O_S=k[X,Y]$ is an extension of $k[X]$ of degree $2$ which coincides with the integral closure of $k[X]$ in $K$. Moreover, as noticed in the proof of Theorem 8, for any rational function $x\in K$, we have the identity
\[\deg_S(f)=\deg(N_{K/\mathcal Q}(x)),\]
where we have set $\mathcal Q=k(X)$. As in the previous section, we simply write $M(K)$ instead of $M_S(K)$.

\begin{theo} The notation and assumptions being as above, we have the identity
\[M(K)=2g.\]
\end{theo}

\begin{proof} Any element of $x\in K$ can be uniquely written as a sum $x=a+Yb$, with $a,b\in\mathcal Q$, and we have the relation
\[N_{K/\mathcal Q}(x)=a^2-fb^2.\]
We want to minimize the degree (meant as $\deg_{\infty}$) of the rational function
\[N_{K/\mathcal Q}(x-y)=(a-c)^2-f(b-d)^2\in\mathcal Q,\]
where $y=c+dY$ is an element of $\mathcal O_S$, with $c,d\in k[X]$. The assumption on the leading coefficient of $f$ leads to the relation
\[\deg\left((a-c)^2-f(b-d)^2\right)=2\max\{\deg(a-c),g+1+\deg(b-d)\}.\]
Once again, we have set $\deg=\deg_\infty$ (in particular, the degree of a rational function is not the number of its poles, counted with their multiplicity). Now, as remarked in the first section, the fact that the ring $k[X]$ is Euclidean with respect to the degree implies that there exist two elements $c,d\in k[X]$ (which are in fact unique) such that $\deg(a-c)\leq-1$ and $\deg(b-d)\leq-1$. We therefore find the inequality
\[\min_{y\in k[X]}\{\deg(N_{K/\mathcal Q}(x-y))\}\leq2g\]
and thus $M(K)\leq2g$. Finally, it is easily checked that for the rational function $x=YX^{-1}\in K$ we have the identity
\[\min_{y\in k[X]}\{\deg(N_{K/\mathcal Q}(x-y))\}=2g,\]
so that $M(K)=2g$, which concludes the proof.
\end{proof}


\begin{thebibliography}{9}

\bibitem[A1963]{a} Armitage, J. Vernon.
{\it Euclid's algorithm in algebraic function fields}.
J. London Math. Soc. {\bf 38} (1963).

\bibitem[BF2006]{b}
Bayer-Fluckiger, Eva.
{\it Upper bounds for  Euclidean minima of algebraic number fields}.
J. Number Theory {\bf 121}, no. 2 (2006).


\bibitem[F1969]{f}
Fulton, William.
{\it Algebraic curves.}
Mathematics Lecture Notes Series (1969).

\bibitem[H1983]{h}
Hartshorne, Robin.
{\it Algebraic geometry}.
GTM {\bf52} (1983).

\bibitem[L1995]{l}
Lemmermeyer, Franz.
{\it The Euclidean algorithm in algebraic number fields}.
Exposition. Math. {\bf 13} (1995), no. 5.

\bibitem[S1988]{s}
Serre, Jean-Pierre.
{\it Algebraic groups and class fields}.
GTM {\bf117} (1988).


\end{thebibliography}
\end{document}